\documentclass[11pt,reqno]{amsart}

\usepackage{amsmath,amssymb,mathtools}
\usepackage{enumitem}
\usepackage{xcolor}
\usepackage[colorlinks=true,
            linkcolor=blue,
            citecolor=blue,
            urlcolor=blue]{hyperref}

\usepackage{etoolbox}

\makeatletter
\patchcmd{\@setabstracta}{\skip@20\p@}{\skip@45\p@}{}{}
\patchcmd{\@maketitle}{\dimen@34\p@}{\dimen@11\p@}{}{}
\makeatother

\allowdisplaybreaks
\numberwithin{equation}{section}
\newtheorem{theorem}{Theorem}[section]

\newtheorem{lemma}[theorem]{Lemma}
\newtheorem{corollary}[theorem]{Corollary}
\newtheorem*{question}{Question}
\theoremstyle{definition}
\newtheorem{definition}[theorem]{Definition}
\theoremstyle{remark}
\newtheorem{remark}[theorem]{Remark}

\newcommand{\Ric}{\operatorname{Ric}}
\newcommand{\Hess}{\operatorname{Hess}}
\newcommand{\Scal}{\operatorname{Scal}}
\newcommand{\tr}{\operatorname{tr}}
\newcommand{\II}{\mathrm{II}}
\newcommand{\Sph}{\mathbb S}
\newcommand{\R}{\mathbb R}
\newcommand{\Z}{\mathbb Z}
\newcommand{\D}{\mathcal D}
\newcommand{\Spin}{\operatorname{Spin}}
\newcommand{\vol}{\operatorname{vol}}

\title{Positive Bakry--\'Emery Ricci Curvature on Homotopy Spheres}
\author{Wen-Qi Li}
	\address{Wen-Qi Li , School of Mathematical Sciences, Key Laboratory of MEA (Ministry of Education) \& Shanghai Key Laboratory of PMMP, East China Normal University, Shanghai 200241, China}
	\email{51255500054@stu.ecnu.edu.cn}
\begin{document}

\begin{abstract}
Wei and Wylie asked whether a complete weighted manifold with
nonnegative Bakry--\'Emery Ricci curvature and bounded potential must
admit a Riemannian metric with nonnegative Ricci curvature.  We answer
this question negatively in every dimension \(8k+1\) and \(8k+2\),
where \(k\geq1\).  Our main geometric result is that every smooth homotopy sphere of
dimension at least seven admits a weighted core metric
\((g,e^{-f})\) with \(\Ric_g+\Hess_g f>0\).  On the other hand, in dimensions \(8k+1\) and \(8k+2\), where
\(k\geq1\), we prove that a homotopy sphere with nonzero
\(\alpha\)-invariant admits no Riemannian metric with \(\Ric\geq0\). Since homotopy spheres with nonzero \(\alpha\)-invariant exist in every dimension \(8k+1\) and \(8k+2\), and since compactness makes the
potential bounded, these manifolds provide the required
counterexamples.
\end{abstract}

\maketitle

\section{Introduction}

Let \((M^n,g,e^{-f}\,d\operatorname{vol}_g)\) be a complete smooth
metric measure space, that is, a complete Riemannian manifold
\((M^n,g)\) equipped with the weighted measure
\(e^{-f}\,d\operatorname{vol}_g\). The associated
\(\infty\)-dimensional Bakry--\'Emery Ricci tensor is
\(\Ric_f:=\Ric_g+\Hess_g f\). The weighted Laplacian is
\(\Delta_f:=\Delta-\langle\nabla f,\nabla\,\cdot\,\rangle\), and it is
formally self-adjoint on compactly supported smooth functions with
respect to the weighted measure. In the Bochner formula for
\(\Delta_f\), the ordinary Ricci tensor is replaced by \(\Ric_f\).
This observation was a starting point for the work of Bakry and
\'Emery on diffusion semigroups and functional
inequalities~\cite{BakryEmery}. The Bakry--\'Emery tensor extends the
ordinary Ricci tensor, and many comparison results remain true when
the potential satisfies suitable bounds.

Bakry--\'Emery curvature is closely connected to the synthetic theory
of lower Ricci curvature bounds. Lott--Villani and Sturm introduced
the curvature-dimension condition \(CD(K,N)\), defined by convexity
of entropy along Wasserstein geodesics
\cite{LottVillani,Sturm}. For \(q\in(0,\infty)\), the
\(q\)-Bakry--\'Emery tensor is
\(\Ric_f^q:=\Ric_g+\Hess_g f-q^{-1}df\otimes df\). On a smooth
weighted Riemannian manifold, the condition \(\Ric_f\geq Kg\)
corresponds to \(CD(K,\infty)\), while \(\Ric_f^q\geq Kg\)
corresponds to \(CD(K,n+q)\). The \(RCD(K,N)\) condition adds a
Hilbertian assumption to the \(CD(K,N)\) condition and rules out
Finsler examples. The relation between the optimal-transport,
Bochner, and heat-flow formulations was developed by
Ambrosio--Gigli--Savar\'e and Erbar--Kuwada--Sturm
\cite{AmbrosioGigliSavare,ErbarKuwadaSturm}. Smooth weighted
Riemannian manifolds are the basic smooth examples in the
\(RCD(K,N)\) theory.

The Bakry--\'Emery tensor also appears naturally in Ricci flow. The
equation \(\Ric_f=\lambda g\) is the gradient Ricci soliton equation,
and gradient Ricci solitons give self-similar solutions to the Ricci
flow introduced by Hamilton~\cite{Hamilton1982}. Perelman's entropy
and reduced-volume arguments showed that gradient shrinking Ricci
solitons are natural models in the study of Ricci-flow singularities
\cite{Perelman1,Perelman2}. These ideas played an important role in
his proof of the Poincar\'e conjecture. We refer to
\cite{CaoSolitons} for a survey of gradient Ricci solitons. Thus
Bakry--\'Emery curvature appears in diffusion theory, optimal
transport, and Ricci flow.

When \(f\) is constant, \(\Ric_f\) is the ordinary Ricci tensor. It is
therefore natural to ask which consequences of a Ricci lower bound
remain true in the weighted setting. Without a bound on \(f\),
classical results such as the Myers compactness theorem, the
Bishop--Gromov volume comparison theorem, and the Cheeger--Gromoll
splitting theorem may fail even when \(\Ric_f\) has the corresponding
lower bound. Wei and Wylie showed that many of these results can be
recovered when the potential is bounded~\cite{WeiWylie}. At the end
of their paper, they asked whether boundedness is also enough to
recover the existence of an ordinary metric with nonnegative Ricci
curvature.

\begin{question}
If \(M\) carries a complete weighted metric with \(\Ric_f\geq0\) and
bounded potential \(f\), must \(M\) admit a Riemannian metric with
\(\Ric\geq0\)?
\end{question}

Wei and Wylie noted that no counterexample was then known even without
the boundedness assumption~\cite[Question~7.5]{WeiWylie}. The
difficulty is that the Hessian term is locally flexible, while the
existence of a metric with nonnegative Ricci curvature is restricted
by global topology. A counterexample must therefore combine a
construction of positive Bakry--\'Emery Ricci curvature with a
topological obstruction to nonnegative ordinary Ricci curvature.
Surgery and gluing methods provide a natural way to search for such
examples.

For positive scalar curvature, the Gromov--Lawson surgery theorem
shows that positivity is preserved under surgeries of codimension at
least three, and related basic results were proved by Schoen and Yau
\cite{SchoenYau,GromovLawson}. There is no surgery theorem of the same
generality for positive Ricci curvature. The Ricci tensor is more
sensitive to the metric in the gluing region, and known constructions
require additional geometric control; see, for example, the work of
Sha--Yang and Wraith~\cite{ShaYang,WraithSurgery}. For connected sums,
Perelman's boundary gluing construction~\cite{Perelman} led Burdick
to introduce core metrics. A core metric gives a useful sufficient
condition for preserving positive Ricci curvature under connected
sum~\cite{Burdick}.

Reiser and Tripaldi recently developed weighted versions of these
ideas~\cite{ReiserTripaldi}. They proved a weighted analogue of
Perelman's gluing theorem, introduced weighted core metrics, and
proved surgery results for positive Bakry--\'Emery Ricci curvature.
An important feature of their higher-surgery theorem is that its
assumptions are local near the surgery sphere. As one application,
they proved that every closed simply connected spin five-manifold
admits a weighted metric with positive Bakry--\'Emery Ricci
curvature. They also noted that, at that time, no manifold was known
to separate the existence of such a weighted metric from the
existence of a Riemannian metric with positive Ricci curvature. Their
weighted gluing and core results provide the main gluing tools used in
this paper.

Our obstruction comes from the \(\alpha\)-invariant.  Throughout this
paper, a homotopy \(n\)-sphere means a closed smooth \(n\)-manifold
that is homotopy equivalent to the standard sphere \(\Sph^n\).
For \(n\geq3\), every such homotopy sphere is spin and has a unique
spin structure, so its \(\alpha\)-invariant is well defined. This invariant is the
Atiyah--Bott--Shapiro orientation
\(\alpha\colon\Omega_n^{\mathrm{Spin}}\to KO_n\)
\cite{AtiyahBottShapiro}; analytically, it is the \(KO\)-valued index
of the real Dirac operator.  Following the standard terminology, a
\emph{Hitchin sphere} is a homotopy sphere with nonzero
\(\alpha\)-invariant.

Milnor exhibited homotopy spheres that are not spin boundaries in
dimensions \(9,10,17,\) and \(18\)
\cite[Theorem~2]{Milnor1965}. Hitchin later used the \(KO\)-valued
Dirac index to show that homotopy spheres with nonzero
\(\alpha\)-invariant exist in every dimension \(8k+1\) and \(8k+2\),
where \(k\geq1\), and that their nonzero \(\alpha\)-invariant
obstructs positive scalar curvature
\cite[pp.~44--45]{Hitchin}. In these dimensions one has
\(KO_n\cong\mathbb Z/2\).

Our geometric construction applies more generally.  Using the weighted
gluing results of Reiser--Tripaldi together with new cylinder and cap
constructions, we prove that every twisted sphere of dimension at
least three admits a weighted core metric with positive
Bakry--\'Emery Ricci curvature.  In particular, every homotopy sphere
of dimension at least seven admits such a metric.  Combining this
construction with the spinorial obstruction gives our main theorem.

\begin{theorem}\label{thm:main}
Let \(n=8k+1\) or \(n=8k+2\), where \(k\geq1\), and let
\(\Sigma^n\) be a homotopy sphere with
\(\alpha(\Sigma)\neq0\) in \(KO_n\cong\Z/2\).  Then:
\begin{enumerate}[label=\textup{(\roman*)}]
\item \(\Sigma\) admits a weighted metric
      \((g,e^{-f}d\vol_g)\) with \(\Ric_f>0\);
\item \(\Sigma\) admits no Riemannian metric with
      \(\Ric\geq0\).
\end{enumerate}
Consequently, \(\Sigma\) provides a counterexample to the
Wei--Wylie question.
\end{theorem}

Since \(\Sigma\) is compact, the metric in
Theorem~\ref{thm:main}\textup{(i)} is complete and its potential is
bounded.  Thus these examples satisfy exactly the assumptions in the
question of Wei and Wylie.  The first dimensions in which Hitchin
spheres occur are
\(9,10,17,18,25,26,\ldots\).  Although many exotic spheres admit
metrics with positive Ricci curvature~\cite{Wraith}, the spheres used
here have nonzero \(\alpha\)-invariant and therefore admit no metric
with positive scalar curvature.  Part~\textup{(ii)} is stronger,
because it also excludes Ricci-flat metrics.

We briefly describe the proof.  By Smale's \(h\)-cobordism
theorem~\cite{Smale1962}, every homotopy sphere \(\Sigma^n\) with
\(n\geq7\) has a twisted-sphere presentation
\(\Sigma_\phi=D^n\cup_\phi D^n\).  Starting from this presentation,
we construct a weighted cylinder joining the two boundary metrics.
On the part of the cylinder where the ordinary Ricci curvature may
fail to be positive, the Hessian of the potential supplies the
required positive tangential term.  We then attach a weighted cap
whose potential satisfies the boundary condition in the weighted
gluing theorem.  The weighted core completion theorem fills the
remaining round boundary by a round hemisphere.  This produces a
weighted core metric with \(\Ric_f>0\) on \(\Sigma_\phi\).

For the topological part, a nonzero \(\alpha\)-invariant forces the
Dirac operator to have a nonzero harmonic spinor.  If
\(\Ric\geq0\), the Schr\"odinger--Lichnerowicz formula shows that this
spinor is parallel and that the metric is Ricci-flat.  The de~Rham
decomposition theorem and the classification of irreducible holonomy
groups admitting parallel spinors then contradict the rational
cohomology of a sphere.

The construction also has a finite-dimensional consequence. For a
closed manifold \(M\), let \(q_{\mathrm{BE}}(M)\) denote the infimum
of the values \(q>0\) for which \(M\) admits a weighted metric with
\(\Ric_f^q>0\). On a closed manifold, a weighted metric with
\(\Ric_f>0\) satisfies \(\Ric_f^q>0\) for every sufficiently large
finite \(q\). Hence \(q_{\mathrm{BE}}(\Sigma)<\infty\) for every
homotopy sphere \(\Sigma\) of dimension at least seven. For a Hitchin
sphere, the obstruction of Reiser--Tripaldi gives
\(4\leq q_{\mathrm{BE}}(\Sigma)<\infty\)
\cite[Proposition~A.2]{ReiserTripaldi}. We leave open whether
\(q_{\mathrm{BE}}(\Sigma)=4\) for every Hitchin sphere.

The paper is organized as follows.
Section~\ref{sec:preliminaries} collects the weighted gluing and core
results, the topological input, the elementary tensor estimates, and
the curvature formulas for metric cylinders.  The cylinder and cap
are constructed in Sections~\ref{sec:zero-cylinder} and
\ref{sec:cap}, and the geometric construction is completed in
Section~\ref{sec:twisted-core}. The topological obstruction,
the proof of the main theorem, and further examples are presented in
Section~\ref{sec:main-proof}.  The final section records the
finite-dimensional consequence and the remaining sharpness question.

\section{Preliminaries}
\label{sec:preliminaries}

We collect the background needed in the proof.  We first fix our
weighted boundary notation and recall the gluing and core results of
Reiser--Tripaldi.  We then state the topological facts and elementary
tensor estimates used later.  Finally, we record the curvature formulas
for a metric cylinder.  These formulas are standard, but we include a
calculation to fix our curvature and boundary sign conventions.

\subsection{Weighted boundary notation}

Let \(N\) be a boundary component of a Riemannian manifold.  All
boundary quantities are computed with respect to the outward unit
normal \(\nu\).  Our convention is
\(\II(X,Y)=g(\nabla_X\nu,Y)\), \(H=\tr\II\), and
\(H^f=H-\nu(f)\).

\subsection{Weighted gluing}

For \(q\in(0,\infty)\), recall that
\(\Ric_f^q=\Ric+\Hess f-q^{-1}df\otimes df\), while
\(\Ric_f^\infty=\Ric_f\).  We use the following gluing result of
Reiser--Tripaldi \cite[Theorem~A]{ReiserTripaldi}.  Their choices of
the outward normal, \(\II\), and \(H^f\) agree with ours.

\begin{theorem}\label{thm:gluing}
Fix \(q\in(0,\infty]\).  Let \((M_i,g_i,e^{-f_i})\), \(i=1,2\), have
\(\Ric_{f_i}^q>0\).  Suppose that boundary components of \(M_1\) and
\(M_2\) are identified by an isometry and that the boundary values of
the potentials agree under this identification.  After pulling back
the tensors from the second boundary, assume that
\(\II_1+\II_2\geq0\) and
\(H^{f_1}_1+H^{f_2}_2\geq0\).  Then the glued manifold admits a smooth
weighted metric with \(\Ric_f^q>0\).  The new metric and potential agree
with the original ones outside an arbitrarily small neighborhood of
the gluing hypersurface.
\end{theorem}

The boundary inequalities are non-strict, so equality is allowed.  If
the restrictions of both potentials to the boundary are constant, one
may add a constant to one potential so that their boundary values
agree.  This changes neither \(df\) nor \(\Hess f\), and therefore does
not change \(\Ric_f^q\).

\subsection{Weighted cores}

The gluing theorem will first give a weighted metric on a punctured
manifold.  To pass from this punctured construction to a closed
manifold, we use the notion of a weighted core metric.

\begin{definition}
Let \(q\in(0,\infty]\).  A closed \(n\)-manifold \(M\) has a
\emph{weighted core metric with respect to \(q\)} if it admits a
weighted metric \((g,e^{-f})\) with \(\Ric_f^q>0\) and an isometric
embedding of the unit round hemisphere \(D^n\hookrightarrow M\) such
that \(f\) is constant on the hemisphere.  When \(q=\infty\), we simply
say that \(M\) has a \emph{weighted core metric}.
\end{definition}

The main advantage of this definition is that the core condition can
be checked on the complement of a disk.  We use the following
characterization of Reiser--Tripaldi.

\begin{theorem}\label{thm:core-completion}
Fix \(q\in(0,\infty]\).  A closed manifold \(M\) has a weighted core
metric with respect to \(q\) if and only if
\(M\setminus\mathring D^n\) admits a weighted metric with
\(\Ric_f^q>0\) such that
\begin{enumerate}[label=\textup{(\alph*)}]
\item the boundary metric is round and the restriction of \(f\) to the
      boundary is constant;
\item \(\II\geq0\) and \(H^f\geq0\) on the boundary.
\end{enumerate}
The same characterization holds if the inequalities in
\textup{(b)} are replaced by \(\II>0\) and \(H^f>0\).
\end{theorem}

This is the punctured form of
\cite[Lemma~4.3]{ReiserTripaldi}.  Although that lemma is stated in
terms of a weighted metric on \(M\) together with an embedded disk,
the proof of \((3)\Rightarrow(1)\) uses only the restriction of the
metric and potential to \(M\setminus\mathring D^n\).  Hence it applies
to the punctured metric constructed below.

We also record the connected-sum result that will be used later.

\begin{theorem}\label{thm:weighted-sum}
Fix \(q\in(0,\infty]\), and let \(M_0^n\) and \(M_1^n\) be closed
manifolds. If \(M_0\) admits a weighted metric with
\(\Ric_f^q>0\), and \(M_1\) admits a weighted core metric with
respect to \(q\), then \(M_0\#M_1\) admits a weighted metric with
\(\Ric_f^q>0\).
\end{theorem}

This is the case of two summands in
\cite[Theorem~B]{ReiserTripaldi}.

\subsection{Topological and spin-geometric facts}

We now turn from the geometric tools to the topological input.  The
first result places homotopy spheres within the scope of our
twisted-sphere construction.  The second provides the spheres needed
for the obstruction to nonnegative Ricci curvature.

\begin{theorem}[{\cite{Smale1962}}]
Every homotopy \(n\)-sphere with \(n\geq7\) is a twisted sphere.  More
precisely, there is a diffeomorphism
\(\phi:\Sph^{n-1}\to\Sph^{n-1}\) such that
\(\Sigma=D^n\cup_\phi D^n\).  For the standard sphere this is immediate,
and for an exotic sphere it follows from the \(h\)-cobordism theorem.
\end{theorem}

Thus, in dimensions \(n\geq7\), it is enough to treat twisted
spheres in order to construct weighted core metrics on homotopy
spheres. To obtain the counterexamples in the
main theorem, we choose homotopy spheres carrying a nontrivial
index-theoretic obstruction.

\begin{theorem}[{\cite[pp.~44--45]{Hitchin}}]
In every dimension \(8k+1\) and \(8k+2\), where \(k\geq1\), there
exists a homotopy sphere with nonzero \(\alpha\)-invariant.
\end{theorem}

The first theorem reduces the geometric part of the argument to the
twisted-sphere construction.  The second supplies the Hitchin spheres
for which the Ricci curvature obstruction will be proved in
Section~\ref{sec:main-proof}.

\subsection{Elementary tensor estimates}
\label{sec:elementary}

We record two elementary facts used in the geometric construction.  If
\(a\) and \(b\) are symmetric \(2\)-tensors, the notation \(a\geq b\)
means that \(a(v,v)\geq b(v,v)\) for every tangent vector \(v\).  In
particular, \(a>0\) means that \(a\) is positive definite.  If \(h\) is
a Riemannian metric, we write
\(|a|_h^2=h^{ik}h^{j\ell}a_{ij}a_{k\ell}\).  The estimate
\(|a|_h\leq C\) implies \(-Ch\leq a\leq Ch\).  Indeed, in an
\(h\)-orthonormal basis, the absolute value of every eigenvalue of
\(a\) is bounded by \(|a|_h\).

\begin{lemma}
\label{lem:uniform-path}
Let \(X\) be closed, let \(h_s\), \(s\in[0,1]\), be a smooth family of
Riemannian metrics, and let \(A_s\) be a smooth family of symmetric
\(2\)-tensors.
\begin{enumerate}[label=\textup{(\roman*)}]
\item There is \(C\geq1\) such that
      \(C^{-1}h_0\leq h_s\leq Ch_0\) for every \(s\in[0,1]\).
\item There is \(C_A<\infty\) such that
      \(-C_Ah_s\leq A_s\leq C_Ah_s\) for every \(s\in[0,1]\).
\item If \(A_s>0\) for every \(s\), then there is \(d>0\) such that
      \(A_s\geq dh_s\) for every \(s\in[0,1]\).
\end{enumerate}
\end{lemma}

\begin{proof}
At each \((x,s)\in X\times[0,1]\), consider the eigenvalues of \(h_s\)
with respect to \(h_0\), and the eigenvalues of \(A_s\) with respect to
\(h_s\).  These eigenvalues depend continuously on \((x,s)\).
Compactness of \(X\times[0,1]\) gives uniform upper and positive lower
bounds for the eigenvalues of \(h_s\), as well as a uniform bound for
the absolute values of the eigenvalues of \(A_s\).  If \(A_s>0\), its
smallest eigenvalue has a positive minimum.  These bounds prove the
three statements.
\end{proof}

\begin{lemma}\label{lem:schur}
Let \(V=\R\oplus E\), and write a symmetric bilinear form on \(V\) as
\(\mathcal A=\left(\begin{smallmatrix}A&B\\B^*&T\end{smallmatrix}\right)\),
where \(A\in\R\), \(B:E\to\R\), and \(T\) is a symmetric bilinear form
on \(E\).  If \(T>0\), then \(\mathcal A>0\) if and only if
\(A-BT^{-1}B^*>0\).  In particular, if \(h\) is an inner product on
\(E\), \(T\geq ch\), \(|B(v)|\leq C|v|_h\), and
\(A>C^2/c\), then \(\mathcal A>0\).
\end{lemma}

\begin{proof}
Let \(b=T^{-1}B^*(1)\).  Then \(B(v)=T(b,v)\), and completing the square
gives
\[
 \mathcal A((z,v),(z,v))
 =T(v+zb,v+zb)+z^2\bigl(A-BT^{-1}B^*\bigr).
\]
This proves the first statement.  If \(T\geq ch\), then
\(T^{-1}\leq c^{-1}h^{-1}\), and hence
\(BT^{-1}B^*\leq C^2/c\).  The last statement follows.
\end{proof}

\subsection{Curvature formulas for metric cylinders}
\label{sec:cylinder-formulas}

Let \(X^m\) be closed, let \(h(t)\) be a smooth family of metrics on
\(X\), and consider the metric \(G=dt^2+h(t)\) on a product of an
interval with \(X\).  A dot denotes differentiation with respect to
\(t\), and \(F=F(t)\).  We use the curvature convention
\(R(U,V)W=\nabla_U\nabla_VW-\nabla_V\nabla_UW-\nabla_{[U,V]}W\).

The corresponding Ricci and Hessian calculations also appear in
\cite[Lemmas~2.4 and~2.5]{ReiserTripaldi}.  We include the calculation
because it fixes the signs used below and makes the later estimates
easier to follow.

\begin{lemma}\label{lem:blocks}
Fix an \(h(t)\)-orthonormal frame \(e_1,\ldots,e_m\).  The normal,
mixed, and tangential blocks of \(\Ric_G+\Hess_G F\) are
\begin{align}
A&=(\Ric_G+\Hess_G F)(\partial_t,\partial_t)
   =-\frac12\tr_h\ddot h+\frac14|\dot h|_h^2+F'',
   \label{eq:block-A}\\
B(v)&=(\Ric_G+\Hess_G F)(\partial_t,v)
   =-\frac12v(\tr_h\dot h)
     +\frac12\sum_{i=1}^m(\nabla^h_{e_i}\dot h)(v,e_i),
   \label{eq:block-B}\\
T(v,w)&=(\Ric_G+\Hess_G F)(v,w) \notag\\
 &=\Ric_h(v,w)-\frac12\ddot h(v,w)
   +\frac12\sum_{i=1}^m\dot h(v,e_i)\dot h(w,e_i)
   -\frac14(\tr_h\dot h)\dot h(v,w)
   +\frac12F'\dot h(v,w).
   \label{eq:block-T}
\end{align}
Moreover, the second fundamental form of the slice
\(\{t\}\times X\), computed with respect to \(+\partial_t\), is
\(\II=\dot h/2\).

If \(q\in(0,\infty)\), the blocks of
\(\Ric_F^q=\Ric_G+\Hess_G F-q^{-1}dF\otimes dF\) are
\(A_q=A-(F')^2/q\), \(B_q=B\), and \(T_q=T\).
\end{lemma}

\begin{proof}
Choose local coordinates \(x^1,\ldots,x^m\) on \(X\), and put
\(x^0=t\).  The Christoffel symbols involving the \(t\)-direction are
\[
 \Gamma^0_{ij}=-\frac12\dot h_{ij},\qquad
 \Gamma^k_{0i}=\Gamma^k_{i0}
       =\frac12h^{k\ell}\dot h_{\ell i},\qquad
 \Gamma^0_{00}=\Gamma^k_{00}=\Gamma^0_{0i}=0.
\]
The symbols \(\Gamma^k_{ij}\) are those of \(h(t)\).  Set
\(K_{ij}=\dot h_{ij}/2\) and \(\mathcal H=\tr_hK\).  At a fixed point,
choose the \(x^i\)-coordinates to be normal for \(h(t)\).  Substitution
of the Christoffel symbols into the coordinate formula for the Ricci
tensor gives
\begin{align*}
\Ric_G(\partial_t,\partial_t)
   &=-\partial_t\mathcal H-|K|_h^2,\\
\Ric_G(\partial_t,\partial_i)
   &=\nabla^jK_{ij}-\nabla_i\mathcal H,\\
\Ric_G(\partial_i,\partial_j)
   &=\Ric_h(\partial_i,\partial_j)-\dot K_{ij}
     -\mathcal H K_{ij}+2K_{ik}K_j{}^k.
\end{align*}

Since
\(\partial_t h^{ij}=-h^{ik}\dot h_{k\ell}h^{\ell j}\), we have
\(\partial_t\mathcal H=\frac12\tr_h\ddot h
-\frac12|\dot h|_h^2\), while
\(|K|_h^2=\frac14|\dot h|_h^2\).  This gives the Ricci part of
\eqref{eq:block-A}.  Substituting \(K=\dot h/2\) and
\(\mathcal H=\tr_h\dot h/2\) into the other two identities gives the
Ricci parts of \eqref{eq:block-B} and \eqref{eq:block-T}.

Because \(F\) depends only on \(t\), its Hessian satisfies
\(\Hess_G F(\partial_t,\partial_t)=F''\),
\(\Hess_G F(\partial_t,v)=0\), and
\(\Hess_G F(v,w)=F'\dot h(v,w)/2\).  Adding these terms proves
\eqref{eq:block-A}--\eqref{eq:block-T}.

Finally,
\(G(\nabla_v\partial_t,w)=\dot h(v,w)/2\), which proves the formula
for \(\II\).  Since \(dF=F'\,dt\), the term
\(q^{-1}dF\otimes dF\) changes only the normal block, giving the stated
formulas for \(\Ric_F^q\).
\end{proof}

\begin{remark}
The last term in \eqref{eq:block-T} is the main positive term used in
the construction.  If \(F'>0\) and \(\dot h>0\), then
\(F'\dot h/2\) is positive definite and can dominate the remaining
tangential terms once uniform bounds are available.
\end{remark}

\section{The cylinder construction}
\label{sec:zero-cylinder}

The next lemma is the central geometric calculation.

\begin{lemma}\label{lem:zero-cylinder}
Let \(X^m\) be closed, and let \(h_0,h_1\) be metrics with positive
Ricci curvature and assume that \(D:=h_1-h_0>0\) as a symmetric
\(2\)-tensor.  Then there are \(L,Q>0\), a metric
\(G=dt^2+h(t)\) on \([0,L]\times X\), and a potential
\(F(t)=Qt^2/2\) such that
\begin{enumerate}[label=\textup{(\roman*)}]
\item \(\Ric_G+\Hess_G F>0\);
\item \(h(0)=h_0\) and \(h(L)=h_1\);
\item both boundary components are totally geodesic;
\item \(F'(0)=0\) and \(F'(L)=QL\).
\end{enumerate}
\end{lemma}

\begin{proof}
For \(s\in[0,1]\), set \(h_s=h_0+sD\).
Every \(h_s\) is a metric because \(D>0\).  The path need not have
positive Ricci curvature in its middle.  We only use positivity near
its two endpoints.  Since positive Ricci curvature is an open
condition, the smallest eigenvalue of \(\Ric_{h_s}\) relative to
\(h_s\) remains positive for \(s\) in a neighborhood of \(0\), and
also for \(s\) in a neighborhood of \(1\).  Compactness of \(X\)
then gives \(\delta\in(0,\frac14)\) and \(\rho>0\) such that
\begin{equation}\label{eq:endpoint-Ricci}
                 \Ric_{h_s}\geq2\rho h_s
 \quad\text{for}\quad
 s\in[0,2\delta]\cup[1-2\delta,1].
\end{equation}

We next choose a parametrization that is constant near the endpoints.
Fix
\(\tau_0\in(0,\frac14)\), and define the standard flat function
\[
 \vartheta(x)=
 \begin{cases}
   0,&x\leq0,\\
   e^{-1/x},&x>0.
 \end{cases}
\]
Then \(\beta(\tau)=\vartheta(\tau-\tau_0)
\vartheta(1-\tau_0-\tau)\) is smooth on \([0,1]\), vanishes on
\([0,\tau_0]\cup[1-\tau_0,1]\), and is strictly positive on
\((\tau_0,1-\tau_0)\).  Define
\[
             \sigma(\tau)
                =\frac{\displaystyle\int_0^\tau\beta(r)\,dr}
                       {\displaystyle\int_0^1\beta(r)\,dr}.
\]
Then \(\sigma:[0,1]\to[0,1]\) is smooth and nondecreasing, is
identically \(0\) near \(0\), is identically \(1\) near \(1\), and
is strictly increasing between those two constant regions.  Thus
\(\sigma\) has all the properties needed below.

For \(L>0\), define \(s(t)=\sigma(t/L)\) and
\(h(t)=h_{s(t)}\).
Then
\begin{equation}\label{eq:h-derivatives}
        \dot h=\frac1L\sigma'(t/L)D,
        \qquad
        \ddot h=\frac1{L^2}\sigma''(t/L)D.
\end{equation}
In particular, \(\dot h=\ddot h=0\) near \(t=0,L\).

Let \(K=s^{-1}([\delta,1-\delta])\).
The compact set
\(\sigma^{-1}([\delta,1-\delta])\) lies in the region where
\(\sigma'>0\).  Therefore there are constants \(a,b>0\), depending
only on \(\sigma\) and \(\delta\), such that on \(K\)
\[
                         t\geq aL,
              \qquad
                         s'(t)\geq\frac bL.
\]
Indeed, one may take
\(a=\min\sigma^{-1}([\delta,1-\delta])>0\) and
\(b=\min_{\sigma^{-1}([\delta,1-\delta])}\sigma'>0\).
Consequently,
\begin{equation}\label{eq:tsprime}
                              t\,s'(t)\geq ab
                              \qquad\text{on }K.
\end{equation}
This estimate explains the scaling of the parameter interval: slowing
the path makes \(\dot h,\ddot h\) small without making the product
\(t\,s'(t)\) small.

We now fix uniform comparison constants.
Because \(X\times[0,1]\) is compact and \(D>0\), there is \(d>0\)
such that
\begin{equation}\label{eq:D-lower}
                                D\geq d h_s
                                \quad\text{for every }s\in[0,1].
\end{equation}
Put \(M_1=\|\sigma'\|_{L^\infty[0,1]}\) and
\(M_2=\|\sigma''\|_{L^\infty[0,1]}\).
By Lemma~\ref{lem:uniform-path}, smoothness and compactness give finite
constants \(C_D,C_\nabla,C_{\tr},C_R\) such that, for all
\(s\in[0,1]\),
\begin{align}
 |D|_{h_s}&\leq C_D,                                             \label{eq:CD}\\
 |\nabla^{h_s}D|_{h_s}&\leq C_\nabla,                              \label{eq:CDer1}\\
 |d(\tr_{h_s}D)|_{h_s}&\leq C_{\tr},                             \label{eq:CDer2}\\
 \Ric_{h_s}&\geq-C_Rh_s.                                        \label{eq:CRic}
\end{align}

Equations \eqref{eq:h-derivatives} and \eqref{eq:CD} give the
pointwise bounds
\begin{equation}\label{eq:hdot-bounds}
       |\dot h|_h\leq\frac{M_1C_D}{L},
       \qquad
       |\ddot h|_h\leq\frac{M_2C_D}{L^2}.
\end{equation}
Write the terms in the tangential block involving derivatives of
\(h\) as
\[
 E=-\frac12\ddot h
   +\frac12\dot h\circ\dot h
   -\frac14(\tr_h\dot h)\dot h,
\]
where \((\dot h\circ\dot h)(v,w)
=\sum_i\dot h(v,e_i)\dot h(w,e_i)\).
In an \(h\)-orthonormal frame,
\[
 |\dot h\circ\dot h|_h
     =|\dot h^2|_{\mathrm{HS}}
     \leq|\dot h|_h^2,
\]
so the triangle inequality, the bound
\(|\tr_h\dot h|\leq\sqrt m\,|\dot h|_h\), and
\eqref{eq:hdot-bounds} give
\begin{equation}\label{eq:E-bound}
                       |E|_h\leq\frac{C_E}{L^2},
\end{equation}
where one may take
\(C_E=\frac12M_2C_D+
(\frac12+\frac{\sqrt m}{4})M_1^2C_D^2\).
Fix \(L\geq\max\{1,\sqrt{C_E/\rho}\}\),
so that
\begin{equation}\label{eq:L-choice}
                              \frac{C_E}{L^2}\leq\rho.
\end{equation}

Outside \(K\), the value of \(s\) lies in
\([0,\delta)\cup(1-\delta,1]\), so
\eqref{eq:endpoint-Ricci} applies.  By
\eqref{eq:E-bound}--\eqref{eq:L-choice}, \(E\geq-\rho h\) there.
Since \(F'=Qt\geq0\) and \(\dot h=s'D\geq0\), the final term in
\eqref{eq:block-T} is nonnegative for every \(Q\geq0\).  Thus
\begin{equation}\label{eq:T-end}
                             T\geq\rho h
                             \qquad\text{outside }K.
\end{equation}
This choice of \(L\) is independent of \(Q\).

On \(K\), equations \eqref{eq:CRic} and \eqref{eq:E-bound}, together
with \(L\geq1\), give \(\Ric_h+E\geq-(C_R+C_E)h\).
Set \(C=C_R+C_E\).  Using \(F'=Qt\), \(\dot h=s'D\),
\eqref{eq:tsprime}, and \eqref{eq:D-lower}, we obtain
\[
 \frac12F'\dot h
       =\frac12Qt\,s'D
       \geq\frac12Qabd\,h
       \qquad\text{on }K.
\]
Therefore
\begin{equation}\label{eq:T-middle}
                      T\geq
                      \left(-C+\frac12Qabd\right)h
                      \qquad\text{on }K.
\end{equation}
It is therefore enough, at this stage, to require
\begin{equation}\label{eq:Q-middle-choice}
                          Q\geq\frac{2(C+1)}{abd}.
\end{equation}
Then the right-hand side of \eqref{eq:T-middle} is at least \(h\).
Combining \eqref{eq:T-end} and \eqref{eq:T-middle}, we get
\begin{equation}\label{eq:T-uniform}
                              T\geq c h
\end{equation}
for \(c=\min\{\rho,1\}>0\).  Increasing \(Q\) later only increases
\(T\), because \(t\dot h\geq0\).

It remains to control the normal and mixed components.  For
\(F(t)=Qt^2/2\), formula \eqref{eq:block-A} gives
\(A=-\frac12\tr_h\ddot h+\frac14|\dot h|_h^2+Q\).
Since \(|\tr_h\ddot h|\leq\sqrt m\,|\ddot h|_h\),
\eqref{eq:hdot-bounds} gives
\begin{equation}\label{eq:A-B-bounds}
                    A\geq Q-\frac{C_A}{L^2},
              \qquad
                    C_A=\frac{\sqrt m}{2}M_2C_D.
\end{equation}

For the mixed block, use \(\dot h=s'D\) in
\eqref{eq:block-B}.  Since \(s'\) is constant in the \(X\)-variables,
\begin{equation}\label{eq:B-expanded}
 B(v)=\frac{s'}2\left[
      -v(\tr_hD)+\sum_i(\nabla^h_{e_i}D)(v,e_i)
      \right].
\end{equation}
The contraction in the second term satisfies
\(\left|\sum_i(\nabla^h_{e_i}D)(v,e_i)\right|
\leq \sqrt m\,|\nabla^hD|_h\,|v|_h\).
Equations \eqref{eq:CDer1}--\eqref{eq:CDer2} and
\(|s'|\leq M_1/L\) therefore give
\begin{equation}\label{eq:B-bound-explicit}
                              |B(v)|\leq\frac{C_B}{L}|v|_h.
\end{equation}
For instance, one may take
\(C_B=\frac{M_1}{2}(C_{\tr}+\sqrt m\,C_\nabla)\).
This calculation also makes explicit that \(B\) contains no derivative
of \(F\).

With respect to \(\R\partial_t\oplus TX\), the tensor
\(\Ric_G+\Hess F\) has block matrix
\(\left(\begin{smallmatrix}A&B\\ B&T\end{smallmatrix}\right)\).
By \eqref{eq:T-uniform}, \(T^{-1}\leq c^{-1}h^{-1}\), and hence
\(B\,T^{-1}B^*\leq C_B^2/(cL^2)\).
In addition to \eqref{eq:Q-middle-choice}, it is enough to choose
\begin{equation}\label{eq:final-Q-choice}
 Q>\max\left\{
       \frac{2(C+1)}{abd},
       \frac{C_A}{L^2}+\frac{C_B^2}{cL^2}
      \right\}.
\end{equation}
Then \(A-BT^{-1}B^*>0\), and the Schur-complement criterion proves
\(\Ric_G+\Hess_G F>0\).  The quantities
\(\delta,\rho,\sigma,d,C_D,C_\nabla,C_{\tr},C_R,C_E,L\) were all fixed
before \(Q\), so none of them depends on \(Q\).

Finally, we verify the boundary data.
The function \(\sigma\) is constant near its endpoints, hence
\(\dot h=0\) there.  By Lemma~\ref{lem:blocks}, both boundary
components are totally geodesic.  More explicitly, the outward normal
is \(-\partial_t\) at \(t=0\) and \(+\partial_t\) at \(t=L\), so the
two second fundamental forms are respectively
\(-\dot h(0)/2=0\) and \(\dot h(L)/2=0\).  Since \(F\) depends only
on \(t\), it is constant on each boundary component, and
\(F'(0)=0\), \(F'(L)=QL\).
\end{proof}

\begin{remark}
The positive tangential term is \(Qt\,s'(t)D/2\).
It vanishes at \(t=0\), as it must because \(F'(0)=0\).  The original
endpoint metric already has positive Ricci curvature, so no additional
Hessian contribution is needed near \(t=0\).  On the region where
\(\Ric_{h_s}\) need not be positive, the product \(t\,s'\) has the
uniform lower bound \(ab\).
\end{remark}

\section{The cap construction}
\label{sec:cap}

The cylinder ends with a generally large normal derivative \(QL\).
The next lemma constructs a cap with matching boundary data.

\begin{lemma}\label{lem:cap}
Let \(m\geq2\).  For every radius \(r>0\) and every
\(\lambda\in\R\), there is a weighted metric
\((g_D,e^{-u})\) on \(D^{m+1}\) such that
\begin{enumerate}[label=\textup{(\roman*)}]
\item \(\Ric_{g_D}+\Hess_{g_D}u>0\);
\item the boundary is a round sphere of radius \(r\);
\item the boundary is totally geodesic;
\item \(u\) is constant on the boundary and
      \(\nu(u)=\lambda\), where \(\nu\) is the outward unit normal.
\end{enumerate}
\end{lemma}

\begin{proof}
We first treat the case \(r=1\), with the outer boundary placed at
\(t=0\).  Set
\[
 p_0:=\min\left\{
       \frac14,\frac{m-1}{4(|\lambda|+m)}
       \right\},
 \qquad
 a:=\frac{p_0^2}{2},
 \qquad
 \eta:=\min\left\{
       \frac14,\frac{m}{4(|\lambda|+1)}
       \right\},
\]
and let \(T=\sqrt{2(1-\eta)/a}\).  On
\([-T,0]\times\Sph^m\), consider
\begin{equation}\label{eq:outer-profile}
 \widetilde\alpha(t)=1-\frac a2t^2,
 \qquad
 \widetilde u(t)=\lambda t,
 \qquad
 \widetilde g=dt^2+\widetilde\alpha(t)^2g_\circ.
\end{equation}
Then \(\widetilde\alpha(-T)=\eta\),
\(\widetilde\alpha(0)=1\),
\(\widetilde\alpha'(0)=0\), and
\(\widetilde\alpha''=-a\).  Moreover,
\(0\leq\widetilde\alpha'(t)\leq p_\eta\) on the annulus, where
\[
 p_\eta:=\widetilde\alpha'(-T)
        =p_0\sqrt{1-\eta}\leq p_0.
\]

For a warped product \(g=dt^2+\alpha(t)^2g_\circ\) and a radial
potential \(u=u(t)\), the radial and tangential eigenvalues of
\(\Ric_g+\Hess_g u\) are
\begin{align}
 \mathcal R_t
   &=-m\frac{\alpha''}{\alpha}+u'',
   \label{eq:cap-radial}\\
 \mathcal R_\theta
   &=-\frac{\alpha''}{\alpha}
     +(m-1)\frac{1-(\alpha')^2}{\alpha^2}
     +u'\frac{\alpha'}{\alpha}.
   \label{eq:cap-tangent}
\end{align}
All mixed components vanish.  For the metric and potential in
\eqref{eq:outer-profile}, the radial eigenvalue is
\(\mathcal R_t=ma/\widetilde\alpha>0\).  Since
\(0<\widetilde\alpha\leq1\) and
\(0\leq\widetilde\alpha'\leq p_0\), the tangential eigenvalue satisfies
\[
 \mathcal R_\theta
 \geq
 \frac a{\widetilde\alpha}
 +\frac{(m-1)(1-p_0^2)-|\lambda|p_0}
        {\widetilde\alpha^2}.
\]
By the definition of \(p_0\),
\[
 (m-1)(1-p_0^2)-|\lambda|p_0
 \geq
 \frac{15}{16}(m-1)-\frac14(m-1)>0.
\]
Thus \(\Ric_{\widetilde g}+\Hess_{\widetilde g}\widetilde u>0\)
throughout the annulus.

We now close the inner end by a round cap.  Put
\(c_\eta=\sqrt{1-\eta^2}\),
\(t_0=-T-\arcsin\eta\), and define
\(\alpha_-(t)=\sin(t-t_0)\) on \([t_0,-T]\).  Equip this cap with
the metric \(g_-=dt^2+\alpha_-(t)^2g_\circ\) and the constant
potential \(u_-\equiv-\lambda T\).  The metric extends smoothly over
the point \(t=t_0\) and has positive Ricci curvature.  At the common
boundary \(t=-T\), we have
\[
 \alpha_-(-T)=\widetilde\alpha(-T)=\eta,\qquad
 \alpha_-'(-T)=c_\eta,\qquad
 \widetilde\alpha'(-T)=p_\eta.
\]
The boundary values of the two potentials also agree because
\(u_-=-\lambda T=\widetilde u(-T)\).

With respect to the outward normals of the two pieces, the boundary
conditions are
\begin{align*}
 \II_{\mathrm{cap}}+\II_{\mathrm{ann}}
   &=\frac{c_\eta-p_\eta}{\eta}\,g_{\partial},\\
 H^{u_-}_{\mathrm{cap}}
   +H^{\widetilde u}_{\mathrm{ann}}
   &=\frac{m(c_\eta-p_\eta)}{\eta}+\lambda.
\end{align*}
Since \(\eta\leq1/4\) and \(p_\eta\leq1/4\), we have
\(c_\eta-p_\eta>1/2\).  Therefore the first expression is positive,
while the second satisfies
\[
 \frac{m(c_\eta-p_\eta)}{\eta}+\lambda
 \geq\frac{m}{2\eta}-|\lambda|>0.
\]
The last inequality follows from
\(\eta\leq m/[4(|\lambda|+1)]\).

Theorem~\ref{thm:gluing} now joins the round cap and the annulus
without changing the metric or potential near the outer boundary
\(t=0\).  At that boundary,
\(\widetilde\alpha(0)=1\),
\(\widetilde\alpha'(0)=0\), and
\(\widetilde u'(0)=\lambda\).  Hence the resulting weighted metric
on \(D^{m+1}\) has positive Bakry--\'Emery Ricci curvature, unit round
and totally geodesic boundary, constant boundary potential, and
outward normal derivative \(\lambda\).

For general \(r>0\), apply the unit-radius construction with boundary
derivative \(r\lambda\), and then multiply the metric by \(r^2\).
Constant scaling leaves \(\Ric+\Hess u\), regarded as a
\((0,2)\)-tensor, unchanged.  It changes the boundary radius from
\(1\) to \(r\) and divides the outward unit normal by \(r\).
Consequently, the new outward normal derivative is
\((r\lambda)/r=\lambda\).  All the required properties follow.
\end{proof}

\section{Weighted core metrics on twisted spheres}
\label{sec:twisted-core}

\begin{theorem}\label{thm:twisted-core}
Let \(n\geq3\), and let
\(\phi:\Sph^{n-1}\to\Sph^{n-1}\) be a diffeomorphism.  Then
\(\Sigma_\phi=D^n\cup_\phi D^n\)
admits a weighted core metric with \(\Ric_f>0\).
\end{theorem}

\begin{proof}
Let \(g_\circ\) be the unit round metric on \(\Sph^{n-1}\).
Because the sphere is compact, \(g_\circ\) and
\(\phi^*g_\circ\) are uniformly equivalent.  More explicitly, the
smallest eigenvalue of \(\phi^*g_\circ\) relative to \(g_\circ\) has
a positive minimum \(\mu>0\).  Choose \(R>0\) so large that
\(R^2\mu>1\).  Then
\begin{equation}\label{eq:D-on-sphere}
                         R^2\phi^*g_\circ-g_\circ>0.
\end{equation}
The metrics \(h_0=g_\circ\) and \(h_1=R^2\phi^*g_\circ\)
both have positive Ricci curvature.  Indeed, \(h_1\) is the pullback
of a scaled round metric, and, as a \((0,2)\)-tensor,
\[
 \Ric_{h_1}
    =\phi^*\Ric_{R^2g_\circ}
    =(n-2)\phi^*g_\circ
    =\frac{n-2}{R^2}h_1>0.
\]
The restriction \(n\geq3\) is used here: it gives \(n-2>0\).
Apply Lemma~\ref{lem:zero-cylinder}.  We obtain the weighted cylinder
\(([0,L]\times\Sph^{n-1},dt^2+h(t),e^{-F})\), where
\(F(t)=Qt^2/2\), with positive Bakry--\'Emery Ricci curvature and
totally geodesic ends.

Apply Lemma~\ref{lem:cap} with \(r=R\) and \(\lambda=-QL\).
The cap boundary is \((\Sph^{n-1},R^2g_\circ)\).
The map \(\phi\) is an isometry from the \(t=L\) boundary of the
cylinder to the cap boundary because \(\phi^*(R^2g_\circ)=h_1\).
Add a constant to the cap potential so that it agrees with \(F(L)\)
on the gluing hypersurface.  If the original cap boundary value is \(c\), replace
the cap potential \(u\) by \(u+F(L)-c\).  This changes neither
\(\Hess u\) nor the prescribed normal derivative.

We now check the boundary conditions for gluing the cylinder to the
cap.  Both boundary components are totally geodesic.  Hence, after
pulling the cap tensor back by \(\phi\),
\[
             \II_{\mathrm{cyl}}+\phi^*\II_{\mathrm{cap}}=0.
\]
At \(t=L\), the outward unit normal of the cylinder is
\(+\partial_t\), and therefore
\[
             H^F_{\mathrm{cyl}}
             =-\nu(F)=-F'(L)=-QL.
\]
By Lemma~\ref{lem:cap} and the choice \(\lambda=-QL\), the outer
boundary of the cap satisfies
\[
             H^u_{\mathrm{cap}}
             =-\nu(u)=-\lambda=QL.
\]
Consequently,
\[
             H^F_{\mathrm{cyl}}+H^u_{\mathrm{cap}}=0.
\]
The boundary values of the potentials agree after adding the constant
chosen above.  Thus all hypotheses of
Theorem~\ref{thm:gluing} are satisfied, and equality in both boundary
conditions in Theorem~\ref{thm:gluing} is allowed. After smoothing, we obtain a metric with positive Bakry--\'Emery Ricci curvature on
\begin{equation}\label{eq:Wphi}
        W_\phi
           =([0,L]\times\Sph^{n-1})\cup_\phi D^n.
\end{equation}
We next check the topology, without assuming that \(\phi\) extends
over a disk.  Let
\[
             D^n_{\mathrm{col}}
                =D^n\cup_{\mathrm{id}}
                   ([0,L]\times\Sph^{n-1})
\]
denote a disk with an external collar, where the \(t=L\) end of the
collar is attached to \(\partial D^n\).  Define a map
\(\Psi:W_\phi\to D^n_{\mathrm{col}}\) by using the identity on the
cap and, on every collar slice, setting
\[
                         \Psi(t,x)=(t,\phi(x)).
\]
At \(t=L\), the point \((L,x)\) is identified in \(W_\phi\) with
\(\phi(x)\) on the cap, while \((L,\phi(x))\) is identified in
\(D^n_{\mathrm{col}}\) with the same cap point \(\phi(x)\).
Therefore \(\Psi\) is well defined at the gluing hypersurface. Using \(y=\phi(x)\) as the boundary coordinate on the cap, the map
\(\Psi\) is the identity in the collar coordinate \(y\). Hence
\(\Psi\) is a diffeomorphism, so \(W_\phi\) is a disk. Thus \(\phi\) was used
only on the collar; no extension of \(\phi\) across the cap was
assumed.

To locate this disk inside the twisted sphere, write
\(\Sigma_\phi=D_A^n\cup_\phi D_B^n\).
Use a collar of \(\partial D_A^n\), with its outer end attached to
\(D_B^n\), for the cylinder in \eqref{eq:Wphi}; use \(D_B^n\) for
the cap.  The complement of this collar and \(D_B^n\) is a smaller
open disk in the interior of \(D_A^n\).  Hence
\[
                       W_\phi\cong
                       \Sigma_\phi\setminus\mathring D^n.
\]
This identification preserves the boundary parametrization induced by
the collar at \(t=0\). The only remaining boundary is \(t=0\).  Its outward normal relative
to \(W_\phi\) is \(-\partial_t\).  Since \(h\) is constant and
\(F'(0)=0\) near this boundary, one has
\(g_{\partial W_\phi}=g_\circ\), \(F|_{\partial W_\phi}=0\),
\(\II=0\), and \(H^F=0\).
The smoothing at the other end of the cylinder does not change these
data, because Theorem~\ref{thm:gluing} changes the metric and
potential only in an arbitrarily small neighborhood of the
\(t=L\) gluing hypersurface.  Thus the boundary metric is round, the boundary value
of the potential is constant, and \(\II=H^F=0\).  All hypotheses of
Theorem~\ref{thm:core-completion} are satisfied on
\(\Sigma_\phi\setminus\mathring D^n\), so core completion produces a
weighted core metric on the closed twisted sphere \(\Sigma_\phi\).
\end{proof}

\begin{remark}
Core completion is essential here.  The weighted gluing theorem may
alter the metric and the potential in a collar of the gluing hypersurface, and hence
does not by itself preserve a prescribed round hemisphere.
Theorem~\ref{thm:core-completion} turns the punctured metric with
\(\II\geq0\) and \(H^f\geq0\) into a closed metric containing an exact
round hemisphere on which the potential is constant.
\end{remark}

\begin{corollary}\label{cor:all-homotopy-spheres}
Every homotopy sphere of dimension \(n\geq7\) admits a weighted core
metric with \(\Ric+\Hess f>0\).
\end{corollary}

\begin{proof}
By the \(h\)-cobordism input in \cite{Smale1962}, every such
homotopy sphere is diffeomorphic to a twisted sphere.  Apply
Theorem~\ref{thm:twisted-core}.  No index-theoretic property of the
sphere is needed for this construction.
\end{proof}

Let \(\Theta_n\) denote the group of oriented homotopy \(n\)-spheres
under connected sum~\cite{KervaireMilnor}. This gives the following stability result.
\begin{theorem}\label{thm:theta-stability}
Let \(n\geq7\), let \(M^n\) be a closed connected smooth manifold, and let
\(\Sigma\in\Theta_n\).  Then
\[
\begin{split}
 M\text{ admits }(g,f)\text{ with }\Ric_g+\Hess_gf>0\\
 \Longleftrightarrow
 M\#\Sigma\text{ admits such a weighted metric}.
\end{split}
\]
\end{theorem}

\begin{proof}
The standard sphere is trivially a twisted sphere, and every exotic
sphere of dimension at least seven is a twisted sphere by the
topological input in Section~\ref{sec:preliminaries}.  Hence
Theorem~\ref{thm:twisted-core} gives a weighted core metric on every
\(\Sigma\in\Theta_n\).

If \(M\) has \(\Ric_f>0\), Theorem~\ref{thm:weighted-sum} gives
\(\Ric_f>0\) on \(M\#\Sigma\).
Conversely, apply the same implication to \(M\#\Sigma\) and the
inverse homotopy sphere \(-\Sigma\), noting that
\((M\#\Sigma)\#(-\Sigma)\cong M\#\Sph^n\cong M\).
\end{proof}

\section{Proof of the main theorem and further examples}
\label{sec:main-proof}

We first record the obstruction needed below.  It is a standard
consequence of Hitchin's spinorial argument and the holonomy
classification for manifolds with parallel spinors
\cite{Hitchin,Wang}.

\begin{lemma}\label{lem:Ric-obstruction}
Let \(n=8k+1\) or \(8k+2\), where \(k\geq1\), and let \(M^n\) be a
closed simply connected spin rational homology sphere.  If
\(\alpha(M)\neq0\), then \(M\) admits no Riemannian metric with
\(\Ric\geq0\).
\end{lemma}

\begin{proof}
Suppose that \(g\) is a Riemannian metric on \(M\) with
\(\Ric_g\geq0\).  Since \(\alpha(M)\neq0\), the \(KO\)-valued index of
the spin Dirac operator is nonzero.  Therefore the Dirac operator has
a nonzero harmonic spinor \(\psi\).  The
Schr\"odinger--Lichnerowicz formula and \(\Scal_g\geq0\) give
\[
0=\int_M\langle\D^2\psi,\psi\rangle\,d\vol_g
 =\int_M\left(
      |\nabla\psi|^2+\frac14\Scal_g|\psi|^2
    \right)d\vol_g.
\]
Both terms in the last integrand are nonnegative.  It follows that
\(\nabla\psi=0\) and \(\Scal_g|\psi|^2=0\) everywhere.  Thus \(\psi\)
is parallel.  Since \(\psi\) is nonzero, its norm is a positive
constant, and hence \(\Scal_g=0\) everywhere.  Because
\(\Ric_g\geq0\) and \(\tr_g\Ric_g=\Scal_g=0\), all eigenvalues of
\(\Ric_g\) vanish.  Therefore \(\Ric_g=0\).

We next show that the holonomy representation is irreducible.  If it
were reducible, the global de~Rham decomposition theorem would split
\(M\) as a nontrivial Riemannian product of complete simply connected
manifolds.  Compactness of \(M\) rules out a nonzero Euclidean factor
and implies that all the remaining factors are compact.  The pullback
of the top-dimensional cohomology class of any factor would then give,
by the K\"unneth theorem, a nonzero cohomology class of intermediate
degree on \(M\).  This contradicts the assumption that \(M\) is a
rational homology sphere.  Hence the holonomy representation is
irreducible.

The metric is not flat.  Indeed, a complete simply connected flat
manifold of positive dimension is isometric to Euclidean space and is
therefore noncompact.  Since \(M\) is simply connected, its full
holonomy group agrees with its restricted holonomy group.  The
classification of irreducible nonflat Riemannian manifolds admitting
a parallel spinor therefore leaves the holonomy groups
\(SU(m)\), \(Sp(m)\), \(G_2\), and \(\Spin(7)\)~\cite{Wang}.

These groups occur in real dimensions \(2m\), \(4m\), \(7\), and \(8\),
respectively.  If \(n=8k+1\geq9\), none of these dimensions is
possible.  If \(n=8k+2\), the groups \(Sp(m)\), \(G_2\), and
\(\Spin(7)\) are excluded by dimension, so the only remaining
possibility is \(SU(n/2)\).  In this case \(M\) carries a nonzero
parallel K\"ahler form.  This form is harmonic and therefore
represents a nonzero class in \(H^2(M;\R)\), contradicting again that
\(M\) is a rational homology sphere.  Thus no metric with
\(\Ric_g\geq0\) can exist.
\end{proof}
We now prove Theorem~\ref{thm:main}.

\begin{proof}[Proof of Theorem~\ref{thm:main}]
Let \(\Sigma\) be as in Theorem~\ref{thm:main}.  Its dimension is at
least nine, so it is a twisted sphere.  Theorem~\ref{thm:twisted-core}
therefore supplies a smooth weighted metric with
\(\Ric_g+\Hess_gf>0\).
A homotopy sphere is closed, simply connected, and an integral
homology sphere.  Moreover,
\(H^2(\Sigma;\Z/2)=0\), so \(w_2(T\Sigma)=0\) and \(\Sigma\) is spin;
since \(H^1(\Sigma;\Z/2)=0\), the spin structure is unique.
Lemma~\ref{lem:Ric-obstruction} shows that it has no
metric with \(\Ric\geq0\).  Finally, compactness makes \(g\) complete
and \(f\) bounded, so the hypotheses in Wei--Wylie Question~7.5 are
satisfied.
\end{proof}

The argument also gives the following connected-sum examples.

\begin{theorem}\label{thm:connected-sum-examples}
Let \(n=8k+1\) or \(8k+2\), with \(k\geq1\), and let \(N^n\) be a
closed simply connected spin rational homology sphere.  Suppose that
\(\alpha(N)=0\) and that \(N\) admits \((g,f)\) with
\(\Ric_g+\Hess_g f>0\).
Let \(\Sigma^n\) be any Hitchin sphere, meaning a homotopy sphere with
\(\alpha(\Sigma)\neq0\).  Then \(N\#\Sigma\) admits a weighted metric
with \(\Ric_f>0\), but no metric with
\(\Ric\geq0\).
\end{theorem}

\begin{proof}
By Theorem~\ref{thm:theta-stability}, \(N\#\Sigma\) admits a weighted
metric with \(\Ric_f>0\).  The spin Dirac index is additive under
connected sum.  Indeed, removing spin disks and inserting the standard
one-handle gives a spin bordism from the disjoint union
\(N\sqcup\Sigma\) to \(N\#\Sigma\); the \(KO\)-index is a spin-bordism
homomorphism.  Therefore
\[
             \alpha(N\#\Sigma)
                  =\alpha(N)+\alpha(\Sigma)
                  =\alpha(\Sigma)\neq0.
\]
Connected sum with a homotopy sphere preserves simple connectivity
and real homology.  Hence \(N\#\Sigma\) is again a simply connected
spin rational homology sphere, and
Lemma~\ref{lem:Ric-obstruction} rules out \(\Ric\geq0\).
\end{proof}

\begin{remark}
If \(N\) already carries a metric with \(\Ric>0\), take \(f\) to be
constant.  The Lichnerowicz theorem gives \(\alpha(N)=0\), so all the
hypotheses of Theorem~\ref{thm:connected-sum-examples} are satisfied.
The standard sphere gives the simplest choice and recovers
Theorem~\ref{thm:main}.
\end{remark}

\section{Finite-dimensional Bakry--\'Emery curvature}
\label{sec:finite-q}

For \(q\in(0,\infty)\), recall that
\[
\Ric_f^q=\Ric_f-\frac{1}{q}df\otimes df.
\]
On an \(n\)-dimensional manifold, this is the Bakry--\'Emery tensor
corresponding to the total curvature-dimension parameter \(N=n+q\).

\begin{lemma}\label{lem:sharp-fixed-q}
Let \(M\) be closed and suppose that \(\Ric_f>0\). Then
\(\Ric_f^q>0\) for every sufficiently large finite \(q\).
\end{lemma}

\begin{proof}
By compactness, there are constants \(\varepsilon>0\) and \(C<\infty\)
such that \(\Ric_f\geq\varepsilon g\) and \(|df|^2\leq C\). Since
\(df\otimes df\leq Cg\), it follows that
\[
\Ric_f^q\geq\left(\varepsilon-\frac{C}{q}\right)g>0
\]
whenever \(q>C/\varepsilon\).
\end{proof}

For a closed manifold \(M\), define
\[
q_{\mathrm{BE}}(M)
:=
\inf\left\{
q>0:
M\text{ admits a weighted metric with }\Ric_f^q>0
\right\},
\]
with the convention that the infimum of the empty set is \(+\infty\).
Pullback by diffeomorphisms shows that \(q_{\mathrm{BE}}\) is a
diffeomorphism invariant.

The set of admissible values of \(q\) is upward closed. Indeed, if
\(\Ric_f^q>0\) and \(q'>q\), then
\[
\Ric_f^{q'}
=
\Ric_f^q+
\left(\frac{1}{q}-\frac{1}{q'}\right)df\otimes df
>0.
\]
Consequently, if \(q_{\mathrm{BE}}(M)<\infty\), then \(M\) admits such
a weighted metric for every \(q>q_{\mathrm{BE}}(M)\). The infimum,
however, need not itself be admissible.

\begin{theorem}\label{thm:finite-q-twisted}
Every homotopy sphere \(\Sigma^n\), \(n\geq7\), admits a weighted core
metric with respect to every sufficiently large finite \(q\).
Consequently,
\[
q_{\mathrm{BE}}(\Sigma)<\infty.
\]
\end{theorem}

\begin{proof}
By Corollary~\ref{cor:all-homotopy-spheres}, \(\Sigma\) admits a
weighted core metric with \(\Ric_f>0\). Applying
Lemma~\ref{lem:sharp-fixed-q} to this metric gives
\(\Ric_f^q>0\) for every sufficiently large finite \(q\). The core
condition is preserved because \(f\) is constant on the round
hemisphere.
\end{proof}

\begin{corollary}\label{cor:effective-hitchin}
Let \(\Sigma\) be a Hitchin sphere. Then \(\Sigma\) admits no weighted
metric with \(\Ric_f^q>0\) for any \(0<q\leq4\), whereas it admits such
a metric for every sufficiently large finite \(q\). Consequently,
\[
4\leq q_{\mathrm{BE}}(\Sigma)<\infty.
\]
Moreover, \(\Sigma\) admits no Riemannian metric with
\(\Ric\geq0\). In contrast,
\[
q_{\mathrm{BE}}(\Sph^n)=0.
\]
In particular, in every dimension in which Hitchin spheres exist,
\(q_{\mathrm{BE}}\) distinguishes the standard smooth sphere from a
Hitchin sphere.
\end{corollary}

\begin{proof}
Suppose that \(\Sigma\) admitted a weighted metric with
\(\Ric_f^q>0\) for some \(q\leq4\). By
\cite[Proposition~A.2]{ReiserTripaldi}, the \(\alpha\)-invariant of
\(\Sigma\) would vanish, contradicting the definition of a Hitchin
sphere. Thus no such metric exists for \(0<q\leq4\).

Existence for every sufficiently large finite \(q\), and hence the
finiteness of \(q_{\mathrm{BE}}(\Sigma)\), follows from
Theorem~\ref{thm:finite-q-twisted}. The obstruction to a Riemannian
metric with nonnegative Ricci curvature is
Lemma~\ref{lem:Ric-obstruction}. Finally, taking \(f\) constant on the
standard round sphere gives \(\Ric_f^q=\Ric>0\) for every \(q>0\), and
therefore \(q_{\mathrm{BE}}(\Sph^n)=0\).
\end{proof}

Notice that Corollary~\ref{cor:effective-hitchin} excludes the endpoint
\(q=4\). Since \(q_{\mathrm{BE}}\) is defined as an infimum, this does
not imply the strict inequality \(q_{\mathrm{BE}}(\Sigma)>4\). Equality
\(q_{\mathrm{BE}}(\Sigma)=4\) would mean that weighted metrics with
\(\Ric_f^q>0\) exist for every \(q>4\), although no such metric exists
at \(q=4\).

In the usual curvature-dimension notation, \(q=N-n\) measures the
excess of the effective dimension over the manifold dimension. Thus
\(q_{\mathrm{BE}}\) may be viewed as an excess-dimension threshold,
while \(n+q_{\mathrm{BE}}\) gives the corresponding total effective
dimension. The preceding corollary shows that this threshold already
detects some smooth-topological information. It remains to understand
whether it can distinguish different exotic smooth structures in a
fixed dimension and how its value varies with the dimension.

\begin{question}\label{rem:q-four}
Does every Hitchin sphere \(\Sigma\) admit a weighted metric with
\(\Ric_f^q>0\) for every \(q>4\), or equivalently,
is \(q_{\mathrm{BE}}(\Sigma)=4\)? More generally, how does
\(q_{\mathrm{BE}}\) vary among exotic homotopy spheres of a fixed
dimension, and how does this behavior depend on the dimension?
\end{question}

\section*{Acknowledgments}

The author would like to express his sincere gratitude to his advisor,
Professor Meng Zhu, for his continued guidance, encouragement, and
support. The author is also grateful to Professor Guofang Wei for
helpful discussions and for her valuable and encouraging comments on
the manuscript.

\section*{Declaration of generative AI and AI-assisted technologies
	in the manuscript preparation process}

During the preparation of this work, the author used OpenAI's ChatGPT
to assist with language editing and with improving the readability,
clarity, and internal consistency of the mathematical exposition.

Beyond language editing, ChatGPT's assistance was limited to the
topological selection and presentation of the examples. In particular,
it suggested considering smooth homotopy spheres with nonzero
$\alpha$-invariant, commonly called Hitchin spheres, and noted that
the $h$-cobordism theorem gives a twisted-sphere presentation
$\Sigma=D^n\cup_\phi D^n$. These observations helped organize the
application of weighted gluing methods to the two disks and their
common boundary. ChatGPT was also used for preliminary checks of
notation and a small number of routine calculations.

All AI-assisted suggestions were reviewed and edited as necessary by
the author before being incorporated into the manuscript. The author
takes responsibility for the final content of the work.

\end{document}